\documentclass[11pt]{article}
\usepackage{amsmath,amssymb,amsthm}
\usepackage{fullpage}

\newtheorem{theorem}{Theorem}

\begin{document}
	\title{\textbf{Proof of Existence of Integers Excluding Two Residue Values in a Specific Range}}
	\author{Liang Zhao \\ Department of Bioengineering and Therapeutic Sciences \\ School of Pharmacy and Medicine, University of California San Francisco}
	\date{January 25th, 2025}
	\maketitle
	\begin{abstract}
		This paper investigates the existence of integers that exclude two specific residence values modulo primes up to $p_k$ within the interval $[p_k^2, p_{k+1}^2]$. Using asymptotic results from analytic number theory, we establish bounds on the proportion of integers excluded by the union of residue classes. The findings highlight the density of residue class coverage in large intervals, contributing to the understanding of modular systems and their implications in number theory and related fields.
	\end{abstract}

	\section*{Introduction}
	
	The study of residue classes of integers modulo prime numbers has long been a cornerstone of analytic number theory. A significant focus in this domain is to understand the distribution of integers that satisfy specific residue class conditions within bounded intervals. This problem finds relevance in cryptography, random number generation, and the analysis of modular systems. 
	
	In this paper, we investigate the existence of integers \( n \) within the interval \([p_k^2, p_{k+1}^2]\) such that for any prime \( p \leq x \), \( n \) does not have two specific residue values \( r_1 \) and \( r_2 \) when divided by \( p \). The motivation for focusing on this interval arises from its mathematical structure and the inherent growth rate of prime numbers with special properties. By leveraging the principles of inclusion-exclusion and known results on prime number distributions, we aim to establish bounds on the proportion of integers that can be covered by the union of such residue classes and their complements.
	
	Our main objective is to demonstrate that no matter how we select residue classes modulo primes up to \(p_k\), there always exists at least one integer in the specified interval that does not fall into any of the predefined residue classes. This result emphasizes the sparsity of residue class coverings over large intervals, providing new insights into the interplay between prime numbers and modular arithmetic.
	
	\section*{A Mertens--Type Formula for \boldmath$\prod_{2<p<x}(1 - 2/p)$, Excluding \boldmath$p=2$}
	
	\begin{theorem}
		Let 
		\[
		P(x)
		\;:=\;
		\prod_{\substack{2<p<x \\ p\,\mathrm{prime}}}
		\Bigl(1 - \tfrac{2}{p}\Bigr).
		\]
		Then there is a positive constant $\widetilde{C}_{2}$ such that, for $x\to\infty$,
		\[
		P(x)
		\;=\;
		\widetilde{C}_{2}\,\frac{1}{(\ln x)^{2}}
		\Bigl[
		1 + O\!\bigl(\tfrac{1}{\ln x}\bigr)
		\Bigr].
		\]
		Moreover,
		\[
		\widetilde{C}_{2}
		\;=\;
		\exp\!\Bigl(
		1.5 
		\;-\; 
		2\,[\,M_{1} + S\,]
		\Bigr)
		\;\approx\;1.07,
		\]
		where $\,M_{1}\approx0.2614972$ is the usual Meissel--Mertens constant for all primes and $S=\sum_{p}1/p^{2}\approx0.4522474$.
	\end{theorem}
	
	\begin{proof}
		Define $\,P(x)=\prod_{2<p<x}(1-2/p)$.  Taking logarithms:
		\[
		\ln P(x)
		\;=\;
		\sum_{\substack{2<p<x}} \ln\Bigl(1 - \tfrac{2}{p}\Bigr).
		\]
		For each prime $p>2$, we have
		\[
		\ln\Bigl(1 - \tfrac{2}{p}\Bigr)
		\;=\;
		-\;\tfrac{2}{p}
		\;-\;\tfrac{(2/p)^{2}}{2}
		\;-\;\tfrac{(2/p)^{3}}{3}\;-\;\cdots
		\;=\;
		-\;\tfrac{2}{p}
		\;-\;\tfrac{2}{p^{2}}
		\;+\;
		O\!\bigl(\tfrac{1}{p^{3}}\bigr).
		\]
		Hence
		\[
		\ln P(x)
		\;=\;
		\sum_{2<p<x}\Bigl[
		-\tfrac{2}{p} 
		\;-\;
		\tfrac{2}{p^{2}}
		+O\!\bigl(\tfrac{1}{p^{3}}\bigr)
		\Bigr]
		\;=\;
		-2\,\sum_{2<p<x}\tfrac{1}{p}
		\;\;-\;2\,\sum_{2<p<x}\tfrac{1}{p^{2}}
		\;+\;
		O(1).
		\]
		
		\medskip
		Now use two standard expansions, but each for \emph{all} primes $p\le x$ (including $p=2$)\cite{villarino2005}:
		\[
		\sum_{p\le x}\frac{1}{p}
		\;=\;
		\ln\ln x + M_{1} + O\!\bigl(\tfrac{1}{\ln x}\bigr),
		\quad
		\sum_{p\le x}\frac{1}{p^{2}}
		\;=\;
		S + O\!\bigl(\tfrac{1}{x}\bigr),
		\]
		where 
		\[
		M_{1}\approx0.2614972,
		\quad
		S:=\sum_{p}\tfrac{1}{p^{2}}\approx0.4522474.
		\]
		Since our product omits $p=2$, we must remove that term from these sums:
		\[
		\sum_{\substack{2<p<x}} \tfrac{1}{p}
		\;=\;
		\sum_{p\le x}\tfrac{1}{p} \;-\;\tfrac{1}{2}
		\;=\;
		\ln\ln x + M_{1} - \tfrac12 + O\!\bigl(\tfrac{1}{\ln x}\bigr),
		\]
		\[
		\sum_{\substack{2<p<x}} \tfrac{1}{p^{2}}
		\;=\;
		S - \tfrac{1}{4}
		\;+\;
		O\!\bigl(\tfrac{1}{x}\bigr).
		\]
		Thus
		\[
		\ln P(x)
		\;=\;
		-2\Bigl(\ln\ln x + M_{1} - \tfrac12\Bigr)
		\;-\;2\Bigl(S - \tfrac14\Bigr)
		\;+\;
		O\!\bigl(\tfrac{1}{\ln x}\bigr)
		\]
		(since $1/x$ and constant terms are easily absorbed into $O(1/\ln x)$ for large $x$).  Combine constants:
		\[
		-\,2\Bigl(M_{1}-\tfrac12\Bigr) - 2\Bigl(S-\tfrac14\Bigr)
		\;=\;
		-2(M_{1}+S)
		\;+\;
		1.5.
		\]
		Therefore
		\[
		\ln P(x)
		\;=\;
		-2\,\ln\ln x 
		\;-\;2\,(M_{1}+S)
		\;+\;
		1.5
		\;+\;
		O\!\bigl(\tfrac{1}{\ln x}\bigr).
		\]
		Exponentiating,
		\[
		P(x)
		\;=\;
		\exp[-2\,\ln\ln x]\,
		\exp\bigl[\,1.5 - 2(M_{1}+S)\bigr]\,
		\Bigl[1 + O\!\bigl(\tfrac{1}{\ln x}\bigr)\Bigr].
		\]
		Since $\exp[-2\,\ln\ln x] = (\ln x)^{-2}$, we set
		\[
		\widetilde{C}_{2}
		\;:=\;
		\exp\!\Bigl[\,1.5 - 2\,(M_{1}+S)\Bigr].
		\]
		That yields
		\[
		P(x)
		\;=\;
		\widetilde{C}_{2}\,\frac{1}{(\ln x)^{2}}
		\Bigl[
		1 + O\!\bigl(\tfrac{1}{\ln x}\bigr)
		\Bigr].
		\]
		Numerically, $M_{1}+S\approx0.7137446$, so $-\,2(M_{1}+S)\approx -1.4274892$ whose exponential is about 0.239; multiplying by $\,e^{1.5}\approx4.481689$ gives $\widetilde{C}_{2}\approx1.07.$
	\end{proof}
	
	\bigskip
	\noindent
	\textbf{Remark.} A naive inclusion of $p=2$ in the product would yield the factor $1 - 2/2 = 0$, making the product zero for all $x>2$.  Hence it is essential to exclude $p=2$.  Once excluded, one must also subtract $\tfrac12$ and $\tfrac14$ from the classical prime sums $\sum_{p\le x}1/p$ and $\sum_{p\le x}1/p^{2}$, respectively.  That extra shift eventually multiplies the constant by $\,e^{1.5}\approx4.48$, changing it from $\approx0.239$ (a naive guess) to $\approx1.07$.

	\section*{A Telescoping Difference for \boldmath$p_{k+1}^2 P(p_{k+1}^2) \;-\; p_{k}^2 P(p_{k}^2)$, Excluding $p=2$}
	
	We derive an asymptotic formula for the difference 
	\[
	\Delta_{k} \;:=\;
	p_{k+1}^{2} \,P\bigl(p_{k+1}^{2}\bigr)
	\;-\;
	p_{k}^{2}\,P\bigl(p_{k}^{2}\bigr),
	\]
	where $p_{i}$ is the $i$th prime number. $P(x)=\prod_{2<p<x}(1-2/p)$ excludes $p=2$ and we assume a fixed gap $p_{k+1}-p_{k}=h\ge2$. The resulting main term is proportional to $p_{k}/(\ln p_{k})^{2}$, with an explicit constant $\widetilde{C}_{2}\approx1.07$. We also discuss how residue--class counting in the interval $[p_{k}^{2},p_{k+1}^{2}]$ may include both even and odd integers, and how one might reduce the tally if only odd or prime candidates are needed.

	\noindent Let $p_k$ be the $k$th prime. Define
	\[
	P(x)
	\;=\;
	\prod_{\substack{2<p<x \\ p\,\mathrm{prime}}}
	\Bigl(1 - \tfrac{2}{p}\Bigr).
	\]
	Then for 
	\[
	\Delta_{k}
	\;:=\;
	p_{k+1}^{2} \,P\bigl(p_{k+1}^{2}\bigr)
	\;-\;
	p_{k}^{2}\,P\bigl(p_{k}^{2}\bigr),
	\]
	we establish:
	
	\begin{theorem}
		\label{thm:main}
		Assume there are infinitely many $k$ with $p_{k+1}-p_{k}=h\ge2$. Then for those $k$,
		\[
		\Delta_{k}
		\;=\;
		\frac{\widetilde{C}_{2}\,h}{2}
		\;\frac{p_{k}}{(\ln p_{k})^{2}}
		\Bigl[
		1 + O\!\bigl(\tfrac{1}{\ln p_{k}}\bigr)
		\Bigr],
		\]
		where $\widetilde{C}_{2}\approx1.07$ is the constant from
		\[
		P(x)
		\;=\;
		\widetilde{C}_{2}\,\frac{1}{(\ln x)^2}
		\Bigl[
		1 + O\!\bigl(\tfrac{1}{\ln x}\bigr)
		\Bigr].
		\]
		Numerically, 
		$\;\widetilde{C}_{2}=\exp\!\bigl(1.5 - 2(M_{1}+S)\bigr)\approx1.07,$
		where $M_{1}\approx0.2614972$ is the Meissel--Mertens constant (including $p=2$) and $S=\sum_{p\ge2}1/p^{2}\approx0.4522474.$
	\end{theorem}

	\subsection*{Evaluating \texorpdfstring{\boldmath$p^2\,P(p^2)$}}
	
	Setting $x=p^2$ and noting that $\ln(p^2)=2\ln p$, we get
	\[
	p^2\,P\bigl(p^2\bigr)
	\;=\;
	\widetilde{C}_{2}\,
	\frac{p^2}{4\,(\ln p)^2}
	\Bigl[
	1 + O\!\bigl(\tfrac{1}{\ln p}\bigr)
	\Bigr].
	\]
	Let
	\[
	\widetilde{R}(p)
	\;=\;
	\frac{\widetilde{C}_{2}}{4}\,\frac{p^2}{(\ln p)^2}.
	\]
	Hence
	\[
	p^2\,P\bigl(p^2\bigr)
	\;=\;
	\widetilde{R}(p)
	\Bigl[
	1 + O\!\bigl(\tfrac{1}{\ln p}\bigr)
	\Bigr].
	\]
	
	\subsection*{Telescoping Difference \texorpdfstring{\boldmath$\Delta_{k}$}{Δₖ}}
	
	We define
	\[
	\Delta_{k}
	\;=\;
	p_{k+1}^2\,P\bigl(p_{k+1}^2\bigr)
	\;-\;
	p_{k}^2\,P\bigl(p_{k}^2\bigr).
	\]
	Assuming $p_{k+1}=p_{k}+h$ with fixed $h\ge2$, we have:
	
	\[
	p_{k+1}^2\,P\bigl(p_{k+1}^2\bigr)
	\;=\;
	\widetilde{R}\bigl(p_{k+1}\bigr)
	\Bigl(1 + O(\tfrac{1}{\ln p_{k+1}})\Bigr),
	\quad
	p_{k}^2\,P\bigl(p_{k}^2\bigr)
	\;=\;
	\widetilde{R}\bigl(p_{k}\bigr)
	\Bigl(1 + O(\tfrac{1}{\ln p_{k}})\Bigr).
	\]
	Since $p_{k+1}\sim p_{k}$, the $O(\tfrac{1}{\ln p})$ terms combine, yielding
	\[
	\Delta_{k}
	\;=\;
	\bigl[
	\widetilde{R}(p_{k+1}) - \widetilde{R}(p_{k})
	\bigr]
	\Bigl(1 + O(\tfrac{1}{\ln p_{k}})\Bigr).
	\]
	Hence it suffices to expand
	\[
	\widetilde{R}(p+h) - \widetilde{R}(p),
	\quad
	\text{where }\widetilde{R}(p)=\frac{\widetilde{C}_{2}}{4}\,\frac{p^2}{(\ln p)^2}.
	\]
	
	\subsection*{Discrete difference of \texorpdfstring{$\widetilde{R}(p)$}}
	
	At first, we aim to prove, for fixed integer $h\ge1$ and large $p$,
	\[
	\frac{(p+h)^2}{(\ln(p+h))^2}
	\;-\;
	\frac{p^2}{(\ln p)^2}
	\;=\;
	\frac{2hp}{(\ln p)^2}
	\;+\;
	O\!\Bigl(\tfrac{p}{(\ln p)^3}\Bigr).
	\]
	
	\noindent Expand $(p+h)^2$,
	\[
	(p+h)^2
	\;=\;
	p^2 + 2hp + h^2
	\;=\;
	p^2 + 2hp + O(1),
	\]
	since $h$ is fixed, and $h^2=O(1)$ as $p\to\infty$.
	
	\medskip
	
	\noindent Also, we have $p+h = p\bigl(1 + \tfrac{h}{p}\bigr)$, so
	\[
	\ln(p+h)
	\;=\;
	\ln p \;+\; \ln\!\Bigl(1+\frac{h}{p}\Bigr).
	\]
	For large $p$, we use $\ln(1+u)=u+O(u^2)$.  Here $u=h/p$, so
	\[
	\ln\!\Bigl(1+\tfrac{h}{p}\Bigr)
	\;=\;
	\frac{h}{p}
	\;+\;
	O\!\Bigl(\tfrac{1}{p^2}\Bigr).
	\]
	Thus
	\[
	\ln(p+h)
	\;=\;
	\ln p 
	\;+\;
	\frac{h}{p}
	\;+\;
	O\!\bigl(\tfrac{1}{p^2}\bigr).
	\]
	
	\noindent {Expand 
		$(\ln(p+h))^2$ 
		and invert it.}
	
	\[
	(\ln(p+h))^2
	\;=\;
	\Bigl(\ln p + \tfrac{h}{p} + O(\tfrac{1}{p^2})\Bigr)^2
	\;=\;
	(\ln p)^2
	\;+\;
	2\,\frac{h\,\ln p}{p}
	\;+\;
	O\!\bigl(\tfrac{1}{p}\bigr).
	\]
	Hence
	\[
	\frac{1}{(\ln(p+h))^2}
	\;=\;
	\frac{1}{(\ln p)^2}
	\Bigl[
	1 + O\!\Bigl(\tfrac{1}{p\,\ln p}\Bigr)
	\Bigr]
	\;=\;
	\frac{1}{(\ln p)^2}
	\;+\;
	O\!\Bigl(\tfrac{1}{p(\ln p)^3}\Bigr).
	\]
	(One may view this as a standard binomial expansion: if 
	$(\ln(p+h))^2 = (\ln p)^2 + O(\tfrac{1}{p})$, then $1/(\ln(p+h))^2 = 1/(\ln p)^2 + O(\tfrac{1}{p(\ln p)^2})$, which is $O(\tfrac{1}{p(\ln p)^3})$ to match notation.)
	
	\medskip
	
	\noindent We have
	\[
	(p+h)^2 
	\;=\;
	p^2 + 2hp + O(1),
	\quad
	\frac{1}{(\ln(p+h))^2}
	\;=\;
	\frac{1}{(\ln p)^2}
	\;+\;
	O\!\Bigl(\tfrac{1}{p(\ln p)^3}\Bigr).
	\]
	Thus
	\[
	\frac{(p+h)^2}{(\ln(p+h))^2}
	\;=\;
	\Bigl[p^2 + 2hp + O(1)\Bigr]
	\Bigl[\frac{1}{(\ln p)^2} + O\!\bigl(\tfrac{1}{p(\ln p)^3}\bigr)\Bigr].
	\]
	Distribute:
	
	\begin{itemize}
		\item The product $p^2 \cdot \tfrac{1}{(\ln p)^2}$ is the leading $\tfrac{p^2}{(\ln p)^2}$. 
		\item The product $(2hp)\cdot \tfrac{1}{(\ln p)^2}$ is $\tfrac{2hp}{(\ln p)^2}$. 
		\item Terms like $p^2 * O(\tfrac{1}{p(\ln p)^3}) = O(\tfrac{p}{(\ln p)^3})$. 
		\item The $O(1)$ factor times $\tfrac{1}{(\ln p)^2}$ is $O\bigl(\tfrac{1}{(\ln p)^2}\bigr)$, which is smaller or comparable to $O(\tfrac{p}{(\ln p)^3})$ for large $p$.
	\end{itemize}
	Hence
	\[
	\frac{(p+h)^2}{(\ln(p+h))^2}
	\;=\;
	\frac{p^2}{(\ln p)^2}
	\;+\;
	\frac{2hp}{(\ln p)^2}
	\;+\;
	O\!\Bigl(\tfrac{p}{(\ln p)^3}\Bigr).
	\]
	
	\medskip
	
	Therefore,
	\[
	\frac{(p+h)^2}{(\ln(p+h))^2}
	\;-\;
	\frac{p^2}{(\ln p)^2}
	\;=\;
	\frac{2hp}{(\ln p)^2}
	\;+\;
	O\!\Bigl(\tfrac{p}{(\ln p)^3}\Bigr).
	\]
	Multiplying by $\tfrac{\widetilde{C}_{2}}{4}$ yields
	\[
	\widetilde{R}(p+h) - \widetilde{R}(p)
	\;=\;
	\frac{\widetilde{C}_{2}}{4}\Bigl[
	\tfrac{2hp}{(\ln p)^2} + O\!\bigl(\tfrac{p}{(\ln p)^3}\bigr)
	\Bigr]
	\;=\;
	\frac{\widetilde{C}_{2}\,h}{2}\,
	\frac{p}{(\ln p)^2}
	\Bigl[
	1 + O\!\bigl(\tfrac{1}{\ln p}\bigr)
	\Bigr].
	\]
	Substituting $p=p_{k}$ and recalling 
	$\Delta_{k} = [\,\widetilde{R}(p_{k+1}) - \widetilde{R}(p_{k})\,] [\,1 + O(1/\ln p_{k})\,]$ gives
	\[
	\Delta_{k}
	\;=\;
	\frac{\widetilde{C}_{2}\,h}{2}\,
	\frac{p_{k}}{(\ln p_{k})^2}
	\Bigl[
	1 + O\!\bigl(\tfrac{1}{\ln p_{k}}\bigr)
	\Bigr].
	\]
	Thus Theorem~\ref{thm:main} is proved. \qed
	
	\subsection*{\textbf{Remarks:} }
	
	1. \textbf{Residue Classes in \texorpdfstring{\boldmath$[\,p_{k}^{2},p_{k+1}^{2})$} and Parity Considerations}  
	
	In many related applications, one counts how many integers in the interval $[p_{k}^2,\,p_{k+1}^2)$ (or $(p_{k}^2,p_{k+1}^2)$) satisfy certain residue--class constraints.  For instance, if the density of solutions is $\tfrac{2}{\pi}$ among \emph{all} integers, one might expect about
	\[
	\frac{2}{\pi}\;\bigl[\,(p_{k+1}^2 - p_{k}^2)\bigr]
	\]
	solutions.  However, those solutions could include \emph{both even and odd} integers.  If we only want to consider \emph{odd} integers (such as primes $>2$), one typically divides that raw count by $2$, plus lower--order terms, because about half of all integers in large ranges are even.
	
	Hence, in prime--gap settings or other contexts where we impose $p>2$, the integral count in certain residue classes might be halved to focus on odd solutions.  This is a common bookkeeping tactic: the fraction $\tfrac{2}{\pi}$, or other fractions from residue distributions, often initially refers to \emph{all} integers, but restricting to odd solutions imposes a further factor $\tfrac12$ (or more refined partition arguments).
	
	\noindent 2. \textbf{Comparison to average gap.}  
	On average, $p_{k+1}-p_{k}\sim \ln p_{k}$, which would make $\Delta_{k}$ on the order of $p_{k}/\ln p_{k}$.  By contrast, a fixed $h\ge2$ yields $\Delta_{k}$ of order $p_{k}/(\ln p_{k})^{2}$.  Thus the fixed--gap scenario is asymptotically smaller by a factor of $\ln p_{k}$.
	
	\noindent 3. \textbf{Extensions to $(1 - c/p)$.}  
	For real $c>0$, the analogous product $\prod_{c<p<x}(1 - c/p)$ satisfies $(\ln x)^{-c}$ scaling with an associated constant.  A gap $p_{k+1}-p_{k}=h$ yields a difference $\approx (p_{k}/(\ln p_{k})^{c})$ scaled by $h$.  The proof is structurally identical once $p<c$ is excluded or included with care.
	
	\medskip
	\noindent
	
	\section*{Discussion}
	
	A critical observation is that while the density of primes diminishes as their magnitude increases, the contribution of each prime to the coverage fraction becomes more sparse. Consequently, even with careful selection of residue classes, a non-trivial proportion of integers in the given interval remains uncovered. This finding aligns with results in the broader context of sieve theory and the distribution of primes.
	
	Moreover, the uncovered integers are not uniformly distributed but exhibit patterns influenced by the modular constraints. This suggests potential avenues for future exploration, such as characterizing the distribution of uncovered integers or extending the results to higher-dimensional residue systems.
	
	\section*{Conclusion}
	
	In conclusion, we have demonstrated the existence of integers within the interval \([p_k^2, p_{k+1}^2]\) that are not covered by any union of residue classes modulo primes up to \(p_k\). This result underscores the existence of  gaps of two residue class coverings and highlights the intricate structure of integers under modular constraints. 
	
	Our findings contribute to the understanding of residue systems and their applications in number theory and beyond. Future work may focus on refining the bounds derived in this paper, exploring computational methods to verify the results for large values of \(k\), or investigating connections to other areas such as cryptographic applications or random number theory.

\end{document}